\documentclass[12pt,reqno]{amsart}
\usepackage{amssymb, amsmath, amsthm, amsfonts,xcolor, enumerate}
\usepackage{verbatim}
\usepackage{hyperref}

\usepackage{amsmath, amsopn, amssymb, amsthm}
\usepackage{enumerate}

\newtheorem{theorem}{Theorem}[section]
\newtheorem{lemma}[theorem]{Lemma}
\newtheorem{cor}[theorem]{Corollary}
\newtheorem{prop}[theorem]{Proposition}

\theoremstyle{definition}
\newtheorem{defn}[theorem]{Definition}

\newcommand{\Kr}{K_{r+1}}

\newcommand{\NN}{\mathbb{N}}
\newcommand{\N}{\mathbb{N}}
\newcommand{\HH}{\mathcal{H}}
\newcommand{\C}{\mathcal{C}}
\newcommand{\I}{\mathcal{I}}
\newcommand{\F}{\mathcal{F}}
\newcommand{\G}{\mathcal{G}}
\newcommand{\Q}{\mathcal{Q}}
\newcommand{\K}{\mathcal{K}}
\newcommand{\U}{\mathcal{U}}
\newcommand{\V}{\mathcal{V}}

\def\le{\leqslant}
\def\ge{\geqslant}
\def\eps{\varepsilon}
\def\ds{\displaystyle}

\title[]{The typical structure of graphs with no large cliques}

\author[J.~Balogh, N.~Bushaw, M.~Collares Neto, H.~Liu, R.~Morris and M.~Sharifzadeh]{J\'ozsef Balogh \and
	Neal Bushaw \and
	Maur\'icio Collares Neto \and
	Hong Liu \and \\
	Robert Morris \and
	Maryam Sharifzadeh}
	
\address{(JB): Department of Mathematical Sciences, University of Illinois at Urbana-Champaign, Urbana, Illinois 61801, USA \and Bolyai Institute, University of Szeged, Szeged, Hungary} \email{jobal@math.uiuc.edu}
\address{(NB): School of Mathematical and Statistical Sciences, Arizona State University, Tempe, AZ 85287 USA \and IMPA, Estrada Dona Castorina 110, Jardim Bot\^anico, Rio de Janeiro, RJ, Brasil} \email{neal@asu.edu}
\address{(MCN, RM): IMPA, Estrada Dona Castorina 110, Jardim Bot\^anico, Rio de Janeiro, RJ, Brasil} 
\email{collares|rob@impa.br}
\address{(HL, MS): Department of Mathematical Sciences, University of Illinois at Urbana-Champaign, Urbana, Illinois 61801, USA} \email{hliu36|sharifz2@illinois.edu}
\thanks{Research supported in part by a Simons Fellowship, NSF CAREER Grant DMS-0745185, Marie Curie FP7-PEOPLE-2012-IIF 327763, Arnold O. Beckman Research Award (UIUC Campus Research Board 13039) (JB), CAPES bolsa Proex (MCN), a CNPq bolsa PDJ (NB) and a CNPq bolsa de Produtividade em Pesquisa (RM)}

\begin{document}

\begin{abstract}
In 1987, Kolaitis, Pr\"omel and Rothschild proved that, for every fixed $r \in \NN$, almost every $n$-vertex $K_{r+1}$-free graph is $r$-partite. In this paper we extend this result to all functions $r = r(n)$ with $r \le (\log n)^{1/4}$. The proof combines a new (close to sharp) supersaturation version of the Erd\H{o}s--Simonovits stability theorem, the hypergraph container method, and a counting technique developed by Balogh, Bollob\'as and Simonovits. 
\end{abstract}

\maketitle

\section{Introduction}\label{sec:intro}

Determining the extremal properties of graphs which avoid a clique of a given size is one of the oldest problems in combinatorics, going back to the early paper of Mantel~\cite{Mantel} and the groundbreaking work of Ramsey~\cite{Ramsey}, Erd\H{o}s and Szekeres~\cite{ESz} and Tur\'an~\cite{Turan} over 70 years ago. The study of the \emph{typical} properties of such graphs was initiated by Erd\H{o}s, Kleitman and Rothschild~\cite{EKR}, who proved in 1976 that almost all triangle-free graphs on $n$ vertices are bipartite.\footnote{That is, the proportion of $n$-vertex triangle-free graphs that are not bipartite goes to zero as $n \to \infty$.} This result was extended to $K_{r+1}$-free graphs, for every fixed $r \in \NN$, ten years later by Kolaitis, Pr\"omel and Rothschild~\cite{KPR}, who showed that almost all such graphs are $r$-partite. Various extensions of this theorem have since been obtained, see for example~\cite{BBS2,PS} for work on other forbidden subgraphs, and~\cite{BMSW,OPT} for a sparse analogue. 

In this paper we extend the result of Kolaitis, Pr\"omel and Rothschild in a different direction, to $K_{r+1}$-free graphs where $r = r(n)$ is a function which is allowed to grow with $n$. More precisely, we prove the following theorem.\footnote{All logs are natural unless otherwise stated.}

\begin{theorem}\label{thm:main}
Let $r = r(n) \in \N_0$ be a function satisfying $r \le (\log n)^{1/4}$ for every $n \in \N$. Then almost all $K_{r+1}$-free graphs on $n$ vertices are $r$-partite.
\end{theorem}

Note that if $r \ge 2 \log_2 n$ then almost all graphs are $K_{r+1}$-free (and almost none are $r$-partite if $r \ll n / \log n$), so the bound on $r$ in Theorem~\ref{thm:main} is not far from being best possible. It would be extremely interesting (and likely very difficult) to determine the largest $\alpha \in [1/4,1]$ such that the theorem holds for some function $r = (\log n)^{\alpha + o(1)}$. It may well be the case that this supremum is equal to $1$, though we are not prepared to state this as a conjecture.

Theorem~\ref{thm:main} improves a recent result of Mousset, Nenadov and Steger~\cite{MNS}, who showed that, for the same\footnote{In fact, a very slightly weaker theorem was stated in~\cite{MNS}, but a little additional case analysis easily gives the result for all $r \le (\log n)^{1/4}$.} family of functions $r = r(n)$, the number of $n$-vertex $K_{r+1}$-free graphs is 
\begin{equation}\label{eq:counting}
2^{t_r(n) + o(n^2/r)},
\end{equation}
where $t_r(n) = \textbf{ex}(n,K_{r+1})$ denotes the number of edges of the Tur\'an graph, the $r$-partite graph on $n$ vertices with the maximum possible number of edges. A bound of this type for fixed $r \in \NN$ was originally proved in~\cite{EKR}, and extended to an arbitrary (fixed) forbidden graph~$H$ in~\cite{EFR}. The problem for $H$-free graphs with $v(H) \to \infty$ as $n \to \infty$ was first studied by Bollob\'as and Nikiforov~\cite{BN}, who proved bounds corresponding to~\eqref{eq:counting} whenever $v(H) = o(\log n)$ and $\chi(H_n) = r+1$ is fixed. For more precise bounds for a fixed forbidden graph $H$, see~\cite{BBS1}, and for similar bounds in the hereditary (i.e., induced-$H$-free) setting, see~\cite{ABBM,BB,BT} and the references therein.

The proof of Theorem~\ref{thm:main} has three main ingredients. The first is the so-called `hypergraph container method', which was recently developed by Balogh, Morris and Samotij~\cite{BMS}, and independently by Saxton and Thomason~\cite{ST}. This method was used by Mousset, Nenadov and Steger to prove Theorem~\ref{thm:mns}, below, from which they deduced the bound~\eqref{eq:counting} using a supersaturation theorem of Lov\'asz and Simonovits~\cite{LS}. 

In order to obtain the much more precise result stated in Theorem~\ref{thm:main}, we will use the method of Balogh, Bollob\'as and Simonovits~\cite{BBS1,BBS2}, who determined the structure of almost all $H$-free graphs for every fixed graph $H$. This powerful technique (see Sections~\ref{sec:properties} and~\ref{proofsec}) allows one to compare the number of $K_{r+1}$-free graphs that are `close' to being $r$-partite, with the total number of $K_{r+1}$-free graphs. 

The missing ingredient is the main new contribution of this paper. In order to deduce from Theorem~\ref{thm:mns} a bound on the number of $K_{r+1}$-free graphs that are `far' from being $r$-partite, we will need an analogue of the Lov\'asz--Simonovits supersaturation result, mentioned above, for the well-known stability theorem of Erd\H{o}s and Simonovits~\cite{ESim}. Although a weak such analogue can easily be obtained via the regularity lemma, this gives bounds which are far from sufficient for our purposes. Instead we will adapt an argument due to F\"uredi~\cite{Furedi} in order to prove the following close-to-best-possible such result. We say that a graph $G$ is \emph{$t$-far from being $r$-partite}\footnote{Similarly, we say that $G$ is $t$-close to being $r$-partite if it is not $t$-far from being $r$-partite.} if $\chi(G') > r$ for every subgraph $G' \subset G$ with $e(G') > e(G) - t$.

  \begin{theorem}\label{supertheorem}
    For every $n,r,t \in \NN$, the following holds. Every graph $G$ on $n$ vertices which is $t$-far from being $r$-partite contains at least
    \begin{equation*}
      \frac{n^{r-1}}{e^{2r} \cdot r!}\left(e(G) + t - \left(1 - \frac{1}{r}\right) \frac{n^2}{2}\right)
    \end{equation*}
    copies of $K_{r+1}$. 
  \end{theorem}
  
  Note that the graph obtained by adding $t$ edges to the Tur\'an graph $T_r(n)$ is $t$-far from being $r$-partite and has roughly $t \cdot (n / r)^{r-1}$ copies of $K_{r+1}$, so Theorem~\ref{supertheorem} is sharp to within a factor of roughly $e^r$. We remark also that the Erd\H{o}s--Simonovits stability theorem for an arbitrary graph $H$ follows from Theorem~\ref{supertheorem} together with the well-known result of Erd\H{o}s~\cite{E64} that the Tur\'an density of any $k$-partite $k$-uniform hypergraph is zero. Indeed, given $c > 0$, every graph $G$ with $n$ vertices and $e(G) \ge t_r(n) - cn^2$ edges that is $2cn^2$-far from being $r$-partite contains at least $\eps n^{r+1}$ copies of $K_{r+1}$ for some $\eps = \eps(c,r) > 0$.  By the result of Erd\H{o}s, it follows that $G$ contains a copy of $K_{r+1}(s)$, the $s$-blow-up of $K_{r+1}$, as long as $n \ge n_0(c,r,s)$ is sufficiently large. We would like to thank Wojciech Samotij for pointing out to us this consequence of Theorem~\ref{supertheorem}.
  
  We will prove Theorem~\ref{supertheorem} in Section~\ref{supersection}, and use it  in Section~\ref{sec:structure} to count the $K_{r+1}$-free graphs that are $n^{2 - 1/r^2}$-far from being $r$-partite. We prove various simple properties of almost all $K_{r+1}$-free graphs in Section~\ref{sec:properties}, and finally, in Section~\ref{proofsec}, we use the Balogh--Bollob\'as--Simonovits method to deduce Theorem~\ref{thm:main}.

\section{A supersaturated Erd\H{o}s-Simonovits stability theorem}\label{supersection}

In this section, we prove our `supersaturated stability theorem' for $\Kr$-free graphs. As noted in the Introduction, we do so by adapting a proof of F\"uredi~\cite{Furedi}. 

Given a graph $G$, a vertex $v \in V(G)$ and an integer $m \in \N$, let us write $K_m(G)$ for the number of $m$-cliques in $G$, and $K_m(v)$ for the number of such $m$-cliques containing~$v$. 

\begin{proof}[Proof of Theorem~\ref{supertheorem}]
We will prove by induction on $r$ that
\begin{equation}\label{eq:super:inductionhypothesis}
K_{r+1}(G) \, \ge \, \frac{n^{r-1}}{c(r)}\left(e(G) + t - \left(1 - \frac{1}{r}\right) \frac{n^2}{2}\right),
\end{equation}
where $c(r) := 2 (r+1)^{r-1} r^{r-1} / r!$, for every graph $G$ on $n$ vertices that is $t$-far from being $r$-partite. Since $c(r) \le e^{2r} r!$, the theorem follows from~\eqref{eq:super:inductionhypothesis}.

Note first that the theorem holds in the case $r = 1$, since a graph is $t$-far from being 1-partite if and only if $e(G) \ge t$, and hence $G$ has at least $\frac{e(G) + t}{2}$ copies of $K_2$, as required. So let $r \ge 2$ and assume that the result holds for $r-1$. Let $n,t \in \NN$, and let $G$ be a graph that is $t$-far from being $r$-partite.  

First, for each $v \in V(G)$, set $B_v = N(v)$ (the set of neighbours of $v$ in $G$) and $A_v = V(G) \setminus B_v$, and observe that
    \begin{equation}\label{eq:super1}
      \sum_{u \in A_v} d(u) = e(G) + e(A_v) - e(B_v),
    \end{equation}
    where $e(X)$ denotes the number of edges in the graph $G[X]$. Now, the graph $G[B_v]$ is $\big( t - e(A_v) \big)$-far from being $(r-1)$-partite, and so, by the induction hypothesis, 
    \begin{equation}\label{eq:super2}
K_{r+1}(v) \, \ge \, \frac{|B_v|^{r-2}}{c(r-1)}\left(e(B_v) + t - e(A_v) - \left(1 - \frac{1}{r-1}\right) \frac{|B_v|^2}{2}\right),
    \end{equation}
since each copy of $K_r$ in $G[B_v]$ corresponds to a copy of $K_{r+1}$ in $G$ that contains $v$. 

Combining~\eqref{eq:super1} and~\eqref{eq:super2}, noting that $|B_v| = d(v)$, and summing over $v$, it follows that 
\begin{equation}\label{eq:super3}(r+1) \cdot K_{r+1}(G) \, \ge \, \sum_{v \in V(G)} \frac{d(v)^{r-2}}{c(r-1)}\bigg( e(G) + t - \sum_{u \in A_v} d(u) - \left(1 - \frac{1}{r-1}\right) \frac{d(v)^2}{2} \bigg).\end{equation}
We claim that
\begin{equation}\label{eq:superclaim}
\sum_{v \in V(G)} \sum_{u \in A_v} d(u) d(v)^{r-2} \le \sum_{v \in V(G)} \sum_{u \in A_v} d(v)^{r-1} = \sum_{v \in V(G)} d(v)^{r-1} \big( n - d(v) \big).
\end{equation}
Indeed, let $X = \big\{ (v,u) : v \in V(G), \, u \in A_v \big\}$ denote the set of ordered pairs in the sum above, and note that $(v,u) \in X$ if and only if $uv \not\in E(G)$. Since $X$ is symmetric, the inequality in~\eqref{eq:superclaim} is in fact an equality for $r = 2$, and for $r = 3$ we apply the Cauchy-Schwarz inequality to obtain
$$\sum_{(v,u) \in X} d(u) d(v) \le \bigg( \sum_{(v,u) \in X} d(u)^2 \bigg)^{1/2} \bigg( \sum_{(v,u) \in X} d(v)^2 \bigg)^{1/2}.$$
For $r \ge 4$, applying H\"older's inequality\footnote{The discrete version of H\"older's inequality states that $\sum_{i=1}^n |x_i y_i| \le \left( \sum_{i=1}^n |x_i|^p\right)^{1/p} \left(\sum_{i=1}^n |y_i|^q\right)^{1/q}$ for $n \in \mathbb{N}$, $x, y \in \mathbb{R}^n$ and $p, q \ge 1$ satisfying $1/p + 1/q = 1$.} with $p = r-2$ and $q = (r-2)/(r-3)$ gives
    \begin{equation*}
      \sum_{(v,u) \in X} d(u) d(v)^{r-2} \le  \bigg(\sum_{(v,u) \in X} d(u)^{r-2} d(v)\bigg)^{1/p} \bigg(\sum_{(v,u) \in X} d(v)^{r-1} \bigg)^{1/q},
    \end{equation*}
since $\big( r - 2 - \frac{1}{r-2} \big) \frac{r-2}{r-3} = \frac{r^2 - 4r + 3}{r-3} = r - 1$. Once again using the symmetry of $X$, and noting that $1 - 1/p = 1/q$, the claimed inequality~\eqref{eq:superclaim} follows. 
    
Combining the inequalities \eqref{eq:super3} and \eqref{eq:superclaim}, we obtain
$$(r+1) \cdot K_{r+1}(G) \, \ge \, \sum_{v \in V(G)} \frac{d(v)^{r-2}}{c(r-1)}\bigg( e(G) + t - d(v) n + \left( 1 + \frac{1}{r-1}\right) \frac{d(v)^2}{2} \bigg).$$
Since the factor in parentheses is minimized when $d(v) = \frac{r-1}{r} \cdot n$, it follows that
$$(r+1) \cdot K_{r+1}(G) \, \ge \, \sum_{v \in V(G)} \frac{d(v)^{r-2}}{c(r-1)}\bigg( e(G) + t - \left( 1 - \frac{1}{r}\right) \frac{n^2}{2} \bigg).$$
Finally, note that every graph $G$ is $\big( e(G) / r \big)$-close to being $r$-partite (take a random partition), and hence we may assume that $\big( 1 + \frac{1}{r} \big) e(G) \ge \big( 1 - \frac{1}{r} \big) \frac{n^2}{2}$, since otherwise the theorem is trivial. Thus, by the convexity of $x^{r-2}$, 
$$\sum_{v \in V(G)} d(v)^{r-2} \ge n \cdot \bigg( \frac{2e(G)}{n} \bigg)^{r-2} \ge \left(\frac{r-1}{r+1}\right)^{r-2} n^{r-1},$$
and so, since $c(r-1) \cdot (r+1)^{r-1} = c(r) \cdot (r-1)^{r-2}$, it follows that
$$K_{r+1}(G) \, \ge \, \frac{n^{r-1}}{c(r)} \bigg( e(G) + t - \left( 1 - \frac{1}{r}\right) \frac{n^2}{2} \bigg),$$
as claimed. 
\end{proof}

\section{An approximate structural result}\label{sec:structure}

In this section we will prove the following approximate version of Theorem~\ref{thm:main}. 

\begin{theorem}\label{thm:structure:approx}
Let $r = r(n) \in \N$ be a function satisfying $r \le (\log n)^{1/4}$ for each $n \in \NN$. Then almost all $\Kr$-free graphs on $n$ vertices are $n^{2-1/r^2}$-close to being $r$-partite.
\end{theorem}

Theorem~\ref{thm:structure:approx} is a straightforward consequence of Theorem~\ref{supertheorem} and the following `container' theorem, which was proved by Mousset, Nenadov and Steger~\cite{MNS} using the hypergraph container method of Balogh, Morris and Samotij~\cite{BMS} and Saxton and Thomason~\cite{ST}. The following theorem is slightly stronger than the result stated in~\cite{MNS}, but follows easily from essentially the same proof. We remark that the deduction of this theorem from the main results of~\cite{BMS,ST} is the only point in the proof of Theorem~\ref{thm:main} where we use our assumption that $r \le (\log n)^{1/4}$, although in Sections~\ref{sec:properties} and~\ref{proofsec} we will require a similar (but somewhat weaker) upper bound on $r$. 

\begin{theorem}\label{thm:mns}
Let $r = r(n) \in \N$ be a function satisfying $r \le (\log n)^{1/4}$. Then there exists a collection $\C$ of graphs such that the following hold for each sufficiently large $n \in \NN$:
\begin{enumerate}
\item[$(a)$] every $\Kr$-free graph on $n$ vertices is a subgraph of some $G \in \C_n$,
\item[$(b)$] $\Kr(G) \le n^{r + 1 - 2/r^2}$ for every $G \in \C_n$, and
\item[$(c)$] $|\C_n| \le \exp\big( n^{2 - 2/r^2} \big)$,
\end{enumerate}
where $\C_n = \big\{ G \in \C : v(G) = n \big\}$. 
\end{theorem}

Deducing Theorem~\ref{thm:structure:approx} from Theorems~\ref{supertheorem} and~\ref{thm:mns} is straightforward.

\begin{proof}[Proof of Theorem~\ref{thm:structure:approx}]
For each $t \in \N$, set 
$$\F_t \, = \, \bigg\{ G \,:\, e(G) \ge \left( 1 - \frac{1}{r} \right) \frac{v(G)^2}{2} - \frac{t}{2} \text{ and $G$ is $t$-far from being $r$-partite} \bigg\},$$
and observe that if $G \in \F_t$, then 
$$K_{r+1}(G) \ge \ds\frac{v(G)^{r-1} \cdot t}{e^{2r+1} \cdot r!},$$ 
by Theorem~\ref{supertheorem}. Therefore, letting $\C$ be the collection of graphs given by Theorem~\ref{thm:mns}, and setting $t = n^{2 - 1/r^2}$, it follows from property~$(b)$ and the bound $r \le (\log n)^{1/4}$ that we have $\C_n \cap \F_t = \emptyset$ for all sufficiently large $n \in \N$. 

Now, for each $\Kr$-free graph $G$ on $n$ vertices that is $n^{2-1/r^2}$-far from being $r$-partite, we have $G \subset C$ for some $C \in \C_n$, and by the observations above and the definition of $\F_t$, it follows that 
$$e(C) \le \left( 1 - \frac{1}{r} \right) \frac{n^2}{2} - \frac{t}{2}.$$
Therefore, summing over all members of $\C_n$, the number of such graphs is at most
$$\exp\big( n^{2 - 2/r^2} \big) \cdot 2^{t_r(n) - t/2} \,\ll\, 2^{t_r(n) - t/4},$$
which is clearly smaller than the number of $\Kr$-free graphs on $n$ vertices, as required.
\end{proof}

\section{Some properties of a typical $\Kr$-free graph}\label{sec:properties}

In this section we will prove some useful structural properties of almost all $\Kr$-free graphs. These structural properties will allow us (in Section~\ref{proofsec}) to count the $\Kr$-free graphs that are close to being $r$-partite, and hence to complete the proof of Theorem~\ref{thm:main}. We emphasize that the lemmas in this section were all proved for fixed $r \in \N$ in~\cite{BBS1}, and no extra ideas are required in order to extend their proofs to our more general setting. 

Let us fix throughout this section a function $2 \le r = r(n) \le (\log n)^{1/4}$, and let us denote by $\G$ the collection of $\Kr$-free graphs on $n$ vertices that are $n^{2-1/r^2}$-close to being $r$-partite, where $n$ is assumed to be sufficiently large. We begin with two simple definitions.

\begin{defn}[Optimal partitions]
An $r$-partition $(U_1,\ldots,U_r)$ of the vertex set of a graph $G$ is called \emph{optimal} if the number of interior edges, $\sum_{i = 1}^r e(U_i)$, is minimized.
\end{defn}

\begin{defn}[Uniformly dense graphs]
We say that a graph $G$ is \emph{uniformly dense} if for every optimal $r$-partition $(U_1,\ldots,U_r)$ and every $i,j \in [r]$ with $i \ne j$, we have 
\begin{equation}\label{eq:reg}
e(A,B) > \frac{|A||B|}{32}
\end{equation}
for every $A \subset U_i$ and $B\subset U_j$ with $|A| = |B| \ge 2^{-10r} n$.
\end{defn}

\begin{lemma}\label{lem:struc:reg}
The number of graphs in $\G$ that are not uniformly dense is at most
$$2^{t_r(n) - 2^{-22r} n^2},$$ 
and therefore almost all $\Kr$-free graphs are uniformly dense.
\end{lemma}

\begin{proof}
In order to count such graphs, we first choose the optimal partition $\U = (U_1,\ldots,U_r)$, the pair $\{i,j\} \subset [r]$, and the sets $A \subset U_i$ and $B\subset U_j$ for which~\eqref{eq:reg} fails. We then choose the edges between $A$ and $B$, and finally the remaining edges. Note first that we have at most $r^n$ choices for $\U$, at most $r^2$ choices for $\{i,j\}$, and at most $2^{2n}$ choices for the pair $(A,B)$. 

Now, the number of choices for the edges between $A$ and $B$ is at most
$$\sum_{k = 0}^{|A||B|/32} \binom{|A||B|}{k} \, \le \, n^2 (32e)^{|A||B|/32} \, \le \, 2^{|A||B|/4},$$
and the number of choices for the remaining edges is at most
$$2^{t_r(n) - |A||B|}\binom{n^2}{n^{2-1/r^2}} \, \le \, 2^{t_r(n) - |A||B|} \exp\Big( n^{2-1/r^2} \log n \Big) \, \le \, 2^{t_r(n) - |A||B|/2},$$
since $\U$ is optimal, $|A||B| \ge 2^{-20r} n^2$, and each $G \in \G$ is $n^{2-1/r^2}$-close to being $r$-partite.

It follows that the number of graphs in $\G$ that are not uniformly dense is at most
$$r^{n+2} \cdot 2^{2n} \cdot 2^{t_r(n) - |A||B|/4} \, \le \, 2^{t_r(n) - 2^{-22r} n^2},$$
as claimed. 
\end{proof}

Our next definition controls the maximum degree inside the parts of an optimal partition. 

\begin{defn}[Internally sparse graphs]\label{def:internal}
A graph $G$ is said to be \emph{internally sparse} if, for every optimal partition $\U = (U_1,\ldots,U_r)$ of $G$, we have
\begin{equation}\label{eq:internal:degrees}
\Delta\big( G[U_i] \big) \, \le \, 2^{-5r} n.
\end{equation}
for every $1 \le i \le r$. Otherwise we say that $G$ is \emph{internally dense}. 
\end{defn}

\begin{lemma}\label{lem:struc:maxdeg}
If $G \in \G$ is internally dense then it is not uniformly dense.
\end{lemma}

We will prove Lemma~\ref{lem:struc:maxdeg} using the following embedding lemma\footnote{In fact, the version stated here is slightly more general than~\cite[Lemma~3.1]{ABKS}, but follows from exactly the same proof.} from~\cite{ABKS}. 

\begin{lemma}\label{lem:embed}
Let $0 < \alpha < 1$, $G$ be a graph, and $W_1,\ldots,W_r \subset V(G)$ be disjoint sets of vertices. Suppose that for every pair $\{i,j\} \subset [r]$ and every pair of sets $A \subset W_i$ and $B \subset W_j$ with $|A| \ge \alpha^r |W_i|$ and $|B| \ge \alpha^r |W_j|$, we have $e(A,B) > \alpha |A||B|$. 

Then $G$ contains a copy of $K_r$ with one vertex in each set $W_j$.  
\end{lemma}

\begin{proof}[Proof of Lemma~\ref{lem:struc:maxdeg}]
Suppose for a contradiction that $G \in \G$ is both internally dense and uniformly dense. Let $\U = (U_1,\ldots,U_r)$ be the optimal partition given by Definition~\ref{def:internal}, and suppose that $v \in U_1$ has degree at least $2^{-5r} n$ in $G[U_1]$. For each $i \in [r]$, let $W_i = N(v) \cap U_i$, and observe that $|W_i| \ge 2^{-5r} n$, since $\U$ is optimal.

Observe that $W_1,\ldots,W_r$ satisfy the conditions of Lemma~\ref{lem:embed} with $\alpha = 1/32$, since $G$ is uniformly dense, so $e(A,B) > |A||B| / 32$ for every pair $\{i,j\} \subset [r]$, and every $A \subset U_i$ and $B\subset U_j$ with $|A| = |B| \ge 2^{-10r} n$. Thus, by Lemma~\ref{lem:embed}, there exists a copy of $K_r$ in the neighborhood of $v$, which (including $v$) gives a copy of $\Kr$ in $G$. But this is a contradiction, since our graph is $\Kr$-free, and so every internally dense graph $G \in \G$ is not uniformly dense, as claimed. 
\end{proof}

Our final definition controls the sizes of the parts in an optimal partition.

\begin{defn}[Balanced graphs] 
A graph $G$ is said to be \emph{balanced} if, for every optimal partition $\U = (U_1,\ldots,U_r)$ of $G$, we have
\begin{equation}\label{eq:well}
\frac{n}{r} - 2^{-5r} n  \, \le \, |U_i| \, \le \, \frac{n}{r} + 2^{-5r} n
\end{equation}
for every $1 \le i \le r$. Otherwise we say that $G$ is \emph{unbalanced}. 
\end{defn}


\begin{lemma}\label{lem:struc:well}
The number of unbalanced graphs in $\G$ is at most
$$2^{t_r(n) - 2^{-12r} n^2},$$ 
and therefore almost all $\Kr$-free graphs are balanced.
\end{lemma}

\begin{proof}
Let $G \in \G$ be an unbalanced graph, and let  $\U = (U_1,\ldots,U_r)$ be an optimal partition of $G$ for which~\eqref{eq:well} fails. Note that
$$\sum_{i = 1}^{r-1} \sum_{j = i + 1}^r |U_i||U_j| \, \le \, t_r(n) - 2^{-11r} n^2,$$
since moving a vertex from a set of size at least $n/r + a$ to one of size $n/r - b$ creates at least $a+b-1$ new potential cross edges. The number of such graphs $G \in \G$ is therefore at most
$$r^n \cdot 2^{t_r(n) - 2^{-11r} n^2} \cdot \binom{n^2}{n^{2-1/r^2}} \, \le \, 2^{t_r(n) - 2^{-12r} n^2},$$
as claimed.
\end{proof}

\section{The proof of Theorem~\ref{thm:main}}\label{proofsec}

In this section we will deduce Theorem~\ref{thm:main} from Theorem~\ref{thm:structure:approx}, using the method of Balogh, Bollob\'as and Simonovits~\cite{BBS1,BBS2}. Recall from the previous section that almost all $\Kr$-free graphs are uniformly dense, internally sparse and balanced. 

Let us fix throughout this section a function $2 \le r = r(n) \le (\log n)^{1/4}$, and assume that $n$ is sufficiently large. 

\begin{defn}
Let $\Q(n,r)$ denote the collection of $\Kr$-free graphs on $n$ vertices that are not $r$-partite, but are $n^{2-1/r^2}$-close to being $r$-partite, and are moreover uniformly dense, internally sparse and balanced.
\end{defn}

Let $\K(n,r)$ denote the collection of $\Kr$-free graphs on $n$ vertices. We will prove the following proposition, which completes the proof of Theorem~\ref{thm:main}.

\begin{prop}\label{prop}
For every sufficiently large $n \in \N$,
$$|\Q(n,r)| \, \le \, 2^{-2^{-6r} n} \cdot |\K(n,r)|.$$
\end{prop}

The idea of the proof is as follows. We will define a collection of bipartite graphs $F_m$ (see Definition~\ref{def:PhimHm}) with parts $\Q(n,r,m)$ and $\K(n,r)$, where the sets $\Q(n,r,m)$ form a partition of $\Q(n,r)$ (see Definitions~\ref{def:mG} and~\ref{def:Qnrtm}). These bipartite graphs will have the following property: the degree in $F_m$ of each $G \in \Q(n,r,m)$ will be significantly larger than the degree of each $G \in \K(n,r)$ (see Lemmas~\ref{lem:outdegrees} and~\ref{lem:indegrees}). The result will then follow by double counting the edges of each $F_m$ and summing over $m$.

In order to define $\Q(n,r,m)$ and $F_m$, we will need the following simple concept.

\begin{defn}[Bad sets] 
Let $G$ be a graph and let $U \subset V(G)$. A set of $r$ vertices $R \subset V(G) \setminus U$ is said to be \emph{bad} towards $U$ if it has no common neighbor in $U$. 
\end{defn}

In the following definition we may choose the partition $\U$ and the sets $X^{(1)},\ldots,X^{(r)}$ arbitrarily, subject to the given conditions.  

\begin{defn}\label{def:mG}
For each $G \in \Q(n,r)$, fix an optimal partition $\U = (U_1,\ldots,U_r)$ of $V(G)$, and for each $j \in [r]$ choose a maximal collection of vertex-disjoint sets $X^{(j)} = \big\{ R^{(j)}_1,\ldots,R^{(j)}_{\ell(j)} \big\}$ such that $R^{(j)}_i$ is bad towards $U_j$ for each $i \in [\ell(j)]$. We define
$$m(G) \, := \, \max\big\{ \ell(j) \,:\, j \in [r] \big\},$$
let $j(G)$ denote the smallest $j$ for which this maximum is attained, and set 
$$X(G) \, := \, R^{\left(j\left(G\right)\right)}_1 \cup \cdots \cup R^{\left(j\left(G\right)\right)}_{\ell\left(j\left(G\right)\right)}.$$
\end{defn}

With this definition in place, it is natural to partition $\Q(n,r)$ by the size of $m(G)$.

\begin{defn}\label{def:Qnrtm}
 For each $m\in\NN$, we define
$$\Q(n,r,m) \,=\, \big\{ G \in \Q(n,r) \,:\, m(G) = m \big\}.$$
\end{defn}

Before continuing, let us note a simple but key fact.

\begin{lemma}\label{lem:mG:lbound}
$m(G) \ge 1$ for every $G \in \Q(n,r)$.
\end{lemma}

\begin{proof}
This follows from the fact that $G$ is not $r$-partite. Indeed, suppose that $m(G) = 0$ and let $x_0x_1 \in E(G[U_1])$ be an `interior' edge of $G$ with respect to $\U$. Since there are no bad $r$-sets towards $U_j$ for any $j \in [r]$, we can recursively choose vertices $x_j \in U_j$ such that $\{x_0,\ldots,x_j\}$ forms a clique. But this is a contradiction, since $G$ is $\Kr$-free. 
\end{proof}

In order to establish an upper bound on those $m$ which we need to consider, we count those graphs in $\Q(n,r)$ for which $m(G)$ is large.

\begin{lemma}\label{lem:mG:ubound}
If $m \ge 2^{-8r} n$, then
$$|\Q(n,r,m)| \, \le \, 2^{t_r(n) - mn / 2^{3r}}.$$
\end{lemma}

\begin{proof}
Let $m \ge 2^{-8r} n$, and consider the number of ways of constructing a graph $G \in \Q(n,r,m)$. We have at most $r^n$ choices for the partition $\U$, at most $\binom{n}{r}^m$ choices for the set $X(G)$, and $r$ choices for $j = j(G)$. Moreover, since $2^r-1 \le 2^r e^{-1/2^r}$, we have at most 
$$2^{t_r(n) - |U_j||X(G)|} \big( 2^r - 1 \big)^{|U_j||X(G)|/r} \, \le \, 2^{t_r(n) - mn / 2^{2r}}$$
choices for the edges between different parts of $\U$, since $X(G)$ is composed of $r$-sets that are bad towards $U_j$, and $G$ is balanced. Finally, we have at most $n^{O(n^{2-1/r^2})}$ choices for the edges inside parts of $\U$, since $G$ is $n^{2-1/r^2}$-close to being $r$-partite.

It follows that
$$\left|\Q(n,r,m)\right| \, \le \, r^n\cdot \binom{n}{r}^m\cdot r\cdot n^{O(n^{2-1/r^2})} \cdot 2^{t_r(n) - mn / 2^{2r}} \, \le \, 2^{t_r(n) - mn / 2^{3r}}$$ 
as required, since $m \ge 2^{-8r} n$, so $n^{2 - 1/r^2} \log n \ll 2^{-3r} m n$. 
\end{proof}

From now on, let us fix a function $1 \le m = m(n) \le 2^{-8r} n$. We are ready to define the bipartite graph $F_m$. 

\begin{defn}\label{def:PhimHm}
Define a map $\Phi_m \colon \Q(n,r,m) \to 2^{\K(n,r)}$ by placing $H \in \Phi_m(G)$ if and only if $H$ can be constructed from $G$ by first removing all edges of $G$ that are incident to $X(G)$, and then adding an arbitrary subset of the edges between $X(G)$ and $V(G) \setminus \big( X(G) \cup U_{j(G)} \big)$.
 
Let $F_m$ be the bipartite graph with edge set $\{ (G,H) : H \in \Phi_m(G) \}$. Moreover, for each $H \in \K(n,r)$, let us write $\Phi_m^{-1}(H) = \{ G \in \Q(n,r,m) : H \in \Phi_m(G) \}$. 
\end{defn}

We first observe that the map $\Phi_m$ is well-defined.

\begin{lemma}
If $G \in \Q(n,r,m)$ and $H \in \Phi_m(G)$, then $H$ is $\Kr$-free. 
\end{lemma}

\begin{proof}
This follows easily from the fact that $G$ is $\Kr$-free, and the maximality of $X(G)$. Indeed, if there exists a copy of $\Kr$ in $H$, then it must contain a vertex of $X(G)$, and therefore it must contain no other vertices of $X(G) \cup U_{j(G)}$. Hence it contains exactly $r$ vertices of $V(G) \setminus \big( X(G) \cup U_{j(G)} \big)$, and by the maximality of $X(G)$ these have a common neighbor in $U_{j(G)}$. But this contradicts our assumption that $G$ is $\Kr$-free, as required.
\end{proof}

We are now ready to prove our first bound on the degrees in $F_m$. 

\begin{lemma}\label{lem:outdegrees}
For every $G \in \Q(n,r,m)$, 
$$\log_2 |\Phi_m(G)| \, \ge \, \bigg( 1 - \frac{1}{r} - \frac{1}{2^{5r}} - \frac{mr}{n} \bigg) mnr.$$   
\end{lemma}

\begin{proof}
This follows immediately from the fact that $G$ is balanced. Indeed, we have two choices for each of the
\begin{equation}\label{eq:counting:outdegrees}
|X(G)| \cdot \big| V(G) \setminus \big( X(G) \cup U_{j(G)} \big) \big| \, \ge \, mr \cdot \bigg( 1 - \frac{1}{r} - \frac{1}{2^{5r}} - \frac{mr}{n} \bigg) n 
\end{equation}
potential edges between $X(G)$ and $V(G) \setminus \big( X(G) \cup U_{j(G)} \big)$.  
\end{proof}

In order to bound the degrees in $F_m$ of vertices in $\K(n,r)$, we will need the following lemma, which counts the optimal partitions in the neighborhood of such a vertex. We note that here, the upper bound on $m$ from Lemma~\ref{lem:mG:ubound} is crucial.

\begin{lemma}\label{lem:diffnumparts}
For each $H \in \K(n,r)$, there are at most $2^{n / 2^{3r}}$ distinct partitions $\U$ of $V(H)$ such that $\U$ is an optimal partition of some graph $G \in \Phi_m^{-1}(H)$.
\end{lemma}

\begin{proof}
We will use the fact that each $G \in \Phi_m^{-1}(H)$ is uniformly dense and $n^{2-1/r^2}$-close to being $r$-partite to show that the optimal partitions in question must be `close' to one another. 

To be precise, let $G_1, G_2 \in \Phi_m^{-1}(H)$, and let $\U = (U_1,\ldots,U_r)$ be an optimal partition of $G_1$ and $\V = (V_1,\ldots,V_r)$ be an optimal partition of $G_2$. We claim that
$$\big| \big\{ j \in [r] \,:\,  |U_i \cap V_j| > 2^{-8r} n + 2mr\big\} \big| \, \le \, 1$$ 
for every $i \in [r]$. Indeed, suppose that 
$$\big| U_i \cap V_{j} \big| > 2^{-8r} n+2mr \qquad \text{and} \qquad \big| U_i \cap V_{j'} \big| > 2^{-8r} n+2mr,$$
set $A = \big( U_i \cap V_j \big) \setminus \big( X(G_1) \cup X(G_2) \big)$ and $B = (U_i \cap V_{j'}) \setminus  \big( X(G_1) \cup X(G_2) \big)$, and note that, since $G_2$ is uniformly dense, we have $e_{G_2}(A,B) > |A||B| / 32 > 2^{-16r-5} n^2$. But these edges are all contained in $U_i$, so this contradicts the fact that $G_1$ is $n^{2-1/r^2}$-close to being $r$-partite, as required.
 
 It follows that (by renumbering the parts if necessary) we have 
$$\big| U_i \setminus V_i \big| \, \le \, r \cdot \big( 2^{-8r} n + 2mr \big) \, \le \, 2^{-6r} n$$
for every $i \in [r]$, where second inequality follows since $m \le 2^{-8r} n$. Set $D_i = U_i \setminus V_i$, and observe that the partition $\V$ and the collection $(D_1,\ldots,D_r)$ together determine $\U$. It follows that the number of optimal partitions is at most
\begin{equation}\label{eq:binom_calc}
  \bigg( \sum_{k = 0}^{2^{-6r} n}\binom{n}{k} \bigg)^r \, \le \, n^r \cdot \binom{n}{2^{-6r}n}^r \, \le \, 2^{r\log n} \cdot \big( e 2^{6r} \big)^{r 2^{-6r} n} \, \le \, 2^{n/2^{3r}},
\end{equation}
as required. 
\end{proof}

We can now bound the degrees on the right. Recall that in Definition~\ref{def:mG} we chose a `canonical' optimal partition for each graph $G \in \Q(n,r)$.

\begin{lemma}\label{lem:indegrees}
We have
$$\log_2 \big| \Phi_m^{-1}(H) \big| \, \le \, \bigg( 1 - \frac{1}{r} - \frac{1}{2^{4r}} \bigg) mnr$$
for every $H \in \K(n,r)$. 
\end{lemma}

\begin{proof}
Let us fix a partition $\U = (U_1,\ldots,U_r)$, and count the number of graphs $G \in \Q(n,r,m)$ with $H \in \Phi_m(G)$ whose optimal partition is $\U$. To do so, first note that we have $\binom{n}{r}^m \le n^{mr}$ choices for $X(G)$, and at most $r$ choices for $j = j(G)$. Now, since $G$ is balanced, i.e., $\big| |U_i| - n/r \big| \le n/2^{5r}$ for each $i \in [r]$, there are at most $(1 - 2/r + 2/2^{5r})n$ possible neighbours for each $v \in X(G)$ not in its own part of $\U$ or in $U_j$. Moreover, since $G$ is internally sparse, each vertex $v \in X(G)$ has at most $2^{-5r} n$ neighbours in its own part of $\U$.  Thus we have at most
$$2^{(1 - 2/r + 2/2^{5r})n} \sum_{k = 0}^{2^{-5r} n}\binom{n}{k}  \, \le \, 2^{(1 - 2/r + 1/2^{3r})n}$$
choices for the edges between each vertex $v \in X(G)$ and $V(G) \setminus U_j$, by bounding as in \eqref{eq:binom_calc}. Finally, by the definition of bad sets, and since $G$ is balanced, we have at most
$$(2^r - 1)^{(1/r + 1/2^{5r})mn} \, \le \, 2^{(1/r + 1/2^{5r})mnr} e^{- mn/r 2^r} \, \le \, 2^{(1/r - 3/2^{3r})mnr}$$
choices for the edges between $X(G)$ and $U_j$.

Since, by Lemma~\ref{lem:diffnumparts}, we have at most $2^{n/2^{3r}}$ choices for the partition $\U$, it follows that
\begin{eqnarray*}  
\log_2 \big| \Phi_m^{-1}(H) \big| \, &\le& \, mr\log n+\log r+\bigg( 1 - \frac{2}{r} + \frac{1}{2^{3r}} + \frac{1}{r} - \frac{3}{2^{3r}} + \frac{1}{2^{3r}} \bigg) mnr \\ 
&\le& \, \bigg( 1 - \frac{1}{r} - \frac{1}{2^{4r}} \bigg) mnr, 
\end{eqnarray*}
as claimed.
\end{proof}

Finally we put the pieces together and prove Proposition~\ref{prop}.

\begin{proof}[Proof of Proposition~\ref{prop}]
We claim first that 
\begin{equation}\label{eq:QmvsK}
|\Q(n,r,m)| \, \le \, 2^{-2^{-5r} mnr} \cdot |\K(n,r)|
\end{equation}
for every $m \le 2^{-8r} n$. To prove this, we simply double count the edges of $F_m$, using Lemmas~\ref{lem:outdegrees} and~\ref{lem:indegrees}. Indeed, we have
$$\log_2 \bigg( \frac{|\Q(n,r,m)|}{|\K(n,r)|} \bigg) \, \le \, \bigg( 1 - \frac{1}{r} - \frac{1}{2^{4r}} \bigg) mnr - \bigg( 1 - \frac{1}{r} - \frac{1}{2^{5r}} - \frac{mr}{n} \bigg) mnr,$$
which implies~\eqref{eq:QmvsK} since $m \le 2^{-8r} n$. 

Summing~\eqref{eq:QmvsK} over $m$, and recalling that $G$ is $n^{2 - 1/r^2}$-close to being $r$-partite, we obtain 
$$\left|\Q\left(n,r\right)\right| \, \le \, 
\sum_{m = 1}^{2^{-8r}n} 2^{-2^{-5r} mnr} \cdot |\K(n,r)| \,+\, \sum_{m = 2^{-8r}n}^n 2^{t_r(n) - mn / 2^{3r}} \, \le \, 2^{-2^{-6r} n} \cdot |\K(n,r)|,$$
by Lemmas~\ref{lem:mG:lbound} and~\ref{lem:mG:ubound} (since $|\K(n,r)| \ge 2^{t_r(n)}$), as required.
\end{proof}

Finally, let us deduce Theorem~\ref{thm:main}.

\begin{proof}[Proof of Theorem~\ref{thm:main}]
By Theorem~\ref{thm:structure:approx}, almost all $\Kr$-free graphs on $n$ vertices are $n^{2-1/r^2}$-close to $r$-partite. We further showed in Lemmas~\ref{lem:struc:reg},~\ref{lem:struc:maxdeg}, and~\ref{lem:struc:well} that almost all of these graphs are either $r$-partite, or in $\Q(n,r)$. Since by Proposition~\ref{prop}, for sufficiently large $n$, the size of $\Q(n,r)$ is (almost) exponentially small compared to $\K(n,r)$, it follows that almost all $\Kr$-free graphs are $r$-partite, as required.
\end{proof}

\section*{Acknowledgements}

The fifth author would like to thank Wojciech Samotij for several interesting discussions about supersaturation and stability, and for showing him F\"uredi's proof of the Erd\H{o}s--Simonovits stability theorem. The authors would also like to thank the referees for their careful reading and valuable comments.


\begin{thebibliography}{99}

\bibitem{ABBM} N.~Alon, J.~Balogh, B.~Bollob\'as and R.~Morris, The structure of almost all graphs in a hereditary property, \emph{J. Combin. Theory, Ser. B}, \textbf{101} (2011), 85--110.

\bibitem{ABKS} N.~Alon, J.~Balogh, P.~Keevash and B.~Sudakov, The number of edge colorings with no monochromatic cliques, \emph{J. London Math. Soc.}, \textbf{70} (2004), 273--288.

\bibitem{AS} N.~Alon and J.H.~Spencer, The probabilistic method, 3rd edition, \emph{Wiley}, New York, 2008.
 
\bibitem{BB} J.~Balogh and J.~Butterfield, Excluding induced subgraphs: critical graphs, \emph{Random Structures Algorithms}, \textbf{38} (2011), 100--120.
 
\bibitem{BBS1} J.~Balogh, B.~Bollob\'as and M.~Simonovits, The number of graphs without forbidden subgraphs, \emph{J. Combin. Theory, Ser. B}, \textbf{91} (2004), 1--24.

\bibitem{BBS2} J.~Balogh, B.~Bollob\'as and M.~Simonovits, The typical structure of graphs without given excluded subgraphs, \emph{Random Structures Algorithms}, \textbf{34} (2009), 305--318.

\bibitem{BMS} J.~Balogh, R.~Morris and W.~Samotij, Independent sets in hypergraphs, {\em J. Amer. Math. Soc}, \textbf{28} (2015), 669--709.

\bibitem{BMSW} J.~Balogh, R.~Morris, W.~Samotij and L.~Warnke, The typical structure of sparse $K_{r+1}$-free graphs, to appear in \emph{Trans. Amer. Math. Soc}.

\bibitem{BN} B.~Bollob\'as and V.~Nikiforov, The number of graphs with large forbidden subgraphs, \emph{European J. Combin.}, \textbf{32} (2010), 1964--1968.

\bibitem{BT} B. Bollob\'as and A. Thomason, Projections of bodies and hereditary properties of hypergraphs, \emph{Bull. London Math. Soc.}, \textbf{27} (1995) 417--424.

\bibitem{E64} P.~Erd\H{o}s, On extremal problems of graphs and generalized graphs, \emph{Israel J. Math.}, \textbf{2} (1964), 183--190.

\bibitem{EFR} P.~Erd\H{o}s, P.~Frankl and V.~R\"odl, The asymptotic number of graphs not containing a fixed subgraph and a problem for hypergraphs having no exponent, \emph{Graphs Combin.}, \textbf{2} (1986), 113--121. 

\bibitem{EKR} P.~Erd\H{o}s, D.J.~Kleitman and B.L.~Rothschild, Asymptotic enumeration of $K_n$-free graphs, Colloquio {I}nternazionale sulle {T}eorie {C}ombinatorie ({R}ome, 1973), {T}omo {II}, 19--27. Atti dei Convegni Lincei, No. 17, Accad. Naz. Lincei, Rome, 1976. 
 
\bibitem{ESim} P.~Erd\H{o}s and M.~Simonovits, A limit theorem in graph theory, \emph{Studia Sci. Math. Hungar.}, \textbf{1} (1966), 51--57.

\bibitem{ESz} P.~Erd\H{o}s and G.~Szekeres, A combinatorial problem in geometry, \emph{Compos. Math.}, \textbf{2} (1935), 463--470.

\bibitem{Furedi} Z.~F\"uredi, A proof of the stability of extremal graphs, Simonovits' stability from Szemer\'edi's regularity, arXiv:1501.03129.

\bibitem{KPR} P. G.~Kolaitis, H. J.~Pr\"omel and B. L.~Rothschild, $K_{\ell+1}$-free graphs: asymptotic structure and a 0-1 law, \emph{Trans. Amer. Math. Soc.}, \textbf{303} (1987), 637--671.

\bibitem{LS} L.~Lov\'asz and M.~Simonovits, On the number of complete subgraphs of a graph, {II}. \emph{Studies in pure mathematics}, Birkhuser (1983), 459--495.

\bibitem{Mantel} W.~Mantel, Problem 28, \emph{Wiskundige Opgaven}, \textbf{10} (1907), 60--61.

\bibitem{MNS} F.~Mousset, R.~Nenadov and A.~Steger, On the number of graphs without large cliques, \emph{Siam J. Discrete Math.}, \textbf{28} (2014), 1980--1986.

\bibitem{OPT} D.~Osthus, H.~J.~Pr\"omel and A.~Taraz, For which densities are random triangle-free graphs almost surely bipartite?, \emph{Combinatorica}, \textbf{23} (2003), 105--150.

\bibitem{PS} H.~J.~Pr\"omel and A.~Steger, The asymptotic number of graphs not containing a fixed color-critical subgraph, \emph{Combinatorica}, \textbf{12} (1992), 463--473.

\bibitem{Ramsey} F.P.~Ramsey, On a problem of formal logic, \emph{{P}roc. {L}ondon {M}ath. {S}oc.}, \textbf{30} (1930), 264--286.

\bibitem{S} M.~Simonovits, A method for solving extremal problems in graph theory, stability problems, \emph{Theory of graphs ({P}roc. {C}olloq., {T}ihany, 1966)}, Academic press, New York, (1968), 279--319.
  
\bibitem{ST} D.~Saxton and A.~Thomason, Hypergraph containers, to appear in \emph{Invent. Math.}

\bibitem{Turan} P.~Tur\'an, Eine Extremalaufgabe aus der Graphentheorie, \emph{Mat. Fiz. Lapok}, \textbf{48} (1941), 436--452.

\end{thebibliography}
\end{document}